\documentclass[11pt]{amsart}   

\usepackage{amssymb,amsthm,amsmath}
\usepackage{color}
\definecolor{darkblue}{rgb}{0, 0, .4}
\definecolor{grey}{rgb}{.7, .7, .7}
\usepackage[breaklinks]{hyperref}
\hypersetup{
	colorlinks=true,
	linkcolor=darkblue,
	anchorcolor=darkblue,
	citecolor=darkblue,
	pagecolor=darkblue,
	urlcolor=darkblue,
	pdftitle={},
	pdfauthor={}
}
	
\newtheorem{theorem}{Theorem}[section]
\newtheorem{lemma}[theorem]{Lemma}

\theoremstyle{definition}
\newtheorem{definition}[theorem]{Definition}
\newtheorem{example}[theorem]{Example}

\theoremstyle{remark}
\newtheorem{remark}[theorem]{Remark}

\numberwithin{equation}{section}

\theoremstyle{theorem}
\newtheorem{corollary}[theorem]{Corollary}
\newtheorem{algorithm}[theorem]{Algorithm}
\newtheorem{proposition}[theorem]{Proposition}

\newcommand{\N}[0]{\mathbb{N}}
\newcommand{\Z}[0]{\mathbb{Z}}

\newcommand{\s}[0]{\sigma}
\newcommand{\f}[0]{\varphi}
\newcommand{\ul}[1]{\underline{#1}}

\newcommand{\union}{\cup}

\newcommand{\hs}{}  
\newcommand{\hd}{\diamond}  
\newcommand{\hz}{\circ}  
\newcommand{\hf}{\bullet}  
\newcommand{\hh}{\star}  

\newcommand{\heap}{ \xymatrix @=-9pt @! } 
\def\SDSize{6}  
\def\SDSizeTANRIGHT{4}  
\def\SDSizeTANLEFT{2}  
\def\SDMidpt{3}  
\def\SDColor{blue}
\def\SDEColor{black}
\newcommand{\StringRXD}[1]{{\color{\SDColor}\xy (\SDSize, \SDSize)*{}; (0, 0)*{}; **\crv{~**\dir{.}(\SDMidpt,\SDMidpt)};  (\SDMidpt, \SDMidpt)*{{\color{\SDEColor}#1}}; \endxy }}
\newcommand{\StringLXD}[1]{{\color{\SDColor}\xy (0, \SDSize)*{}; (\SDSize, 0)*{}; **\crv{~**\dir{.}(\SDMidpt,\SDMidpt)};  (\SDMidpt, \SDMidpt)*{{\color{\SDEColor}#1}}; \endxy }}
\newcommand{\StringL}[1]{{\color{\SDColor}\xy (\SDSize, \SDSize)*{}; (0, 0)*{}; (0, \SDSize)*{}; **\crv{(\SDSizeTANLEFT,\SDMidpt)};  (\SDMidpt, \SDMidpt)*{{\color{\SDEColor}#1}}; \endxy}}
\newcommand{\StringR}[1]{{\color{\SDColor}\xy (0,0)*{}; (\SDSize, 0)*{}; (\SDSize, \SDSize)*{}; **\crv{(\SDSizeTANRIGHT,\SDMidpt)}; (\SDMidpt, \SDMidpt)*{{\color{\SDEColor}#1}}; \endxy }}
\newcommand{\StringLX}[1]{{\color{\SDColor}\xy (0, \SDSize)*{}; (\SDSize, 0)*{}; **\crv{(\SDMidpt,\SDMidpt)};  (\SDMidpt, \SDMidpt)*{{\color{\SDEColor}#1}}; \endxy }}
\newcommand{\StringRX}[1]{{\color{\SDColor}\xy (\SDSize, \SDSize)*{}; (0, 0)*{}; **\crv{(\SDMidpt,\SDMidpt)};  (\SDMidpt, \SDMidpt)*{{\color{\SDEColor}#1}}; \endxy }}
\newcommand{\StringLR}[1]{{\color{\SDColor}\xy (0, 0)*{}; (0, \SDSize)*{}; **\crv{(\SDSizeTANLEFT,\SDMidpt)};  (\SDSize, 0)*{}; (\SDSize, \SDSize)*{}; **\crv{( \SDSizeTANRIGHT,\SDMidpt)}; (\SDMidpt, \SDMidpt)*{{\color{\SDEColor}#1}}; \endxy }}
\newcommand{\StringLRX}[1]{{\color{\SDColor}\xy (0, \SDSize)*{}; (\SDSize, 0)*{}; **\crv{(\SDMidpt,\SDMidpt)};  (\SDSize, \SDSize)*{}; (0, 0)*{}; **\crv{(\SDMidpt,\SDMidpt)}; (\SDMidpt, \SDMidpt)*{{\color{\SDEColor}#1}}; \endxy }}

\usepackage[all, dvips, knot]{xy} 
\usepackage[all, knot, dvips, color]{xypic} 
\xyoption{arc}

\newcommand{\h}{\mathsf{h}}
\newcommand{\w}{\mathsf{w}}

\renewcommand{\u}{\mathsf{u}} 
\renewcommand{\v}{\mathsf{v}} 

\newcommand{\E}{\mathcal{E}}
\newcommand{\block}[1]{\begin{bmatrix} #1 \end{bmatrix}}
\newcommand{\pblock}[1]{\begin{pmatrix} #1 \end{pmatrix}}

\begin{document}

\title{Kazhdan--Lusztig polynomials for maximally-clustered hexagon-avoiding permutations}

\begin{abstract}
We provide a non-recursive description for the bounded admissible sets of masks
used by Deodhar's algorithm \cite{d} to calculate the Kazhdan--Lusztig
polynomials $P_{x,w}(q)$ of type $A$, in the case when $w$ is hexagon avoiding
\cite{b-w} and maximally clustered \cite{losonczy}.  This yields a
combinatorial description of the Kazhdan--Lusztig basis elements of the Hecke
algebra associated to such permutations $w$.  The maximally-clustered
hexagon-avoiding elements are characterized by avoiding the seven classical
permutation patterns $\{3421, 4312, 4321, 46718235, 46781235, 56718234,
56781234\}$.  We also briefly discuss the application of heaps to permutation
pattern characterization.
\end{abstract}

\author{Brant C. Jones}
\address{Department of Mathematics Box 354350, University of Washington, Seattle, WA 98195}
\email{\href{mailto:brant@math.washington.edu}{\texttt{brant@math.washington.edu}}}
\urladdr{\url{http://www.math.washington.edu/\~brant/}}

\thanks{The author received support from NSF grant DMS-9983797.}

\keywords{Kazhdan--Lusztig polynomial, pattern avoidance, 2-sided weak Bruhat
order, 321-hexagon, freely-braided, maximally-clustered.}

\date{\today}

\maketitle




\bigskip
\section{Introduction}\label{s:background}

The Kazhdan--Lusztig polynomials introduced in \cite{k-l} have interpretations
for finite Weyl groups as Poincar\'e polynomials for intersection cohomology of
Schubert varieties \cite{K-L2} and as a $q$-analogue of the multiplicities for
Verma modules \cite{BeilBern,BryKash}.  From these interpretations, it is known
that the Kazhdan--Lusztig polynomials have nonnegative integer coefficients.
However, their combinatorial structure remains obscure and no simple all
positive formula for the coefficients is known in general.  For an introduction
to these polynomials, consult \cite{h,Deodhar94,brenti04,b-b}.

Deodhar \cite{d} has proposed a framework for determining the Kazhdan--Lusztig
polynomials which can be described for an arbitrary Coxeter group.  The
framework gives the Kazhdan--Lusztig polynomials in the form of a combinatorial
generating function, but generally involves summation over a certain recursively
defined set.

In this paper, we show that when $w$ is a permutation that is hexagon avoiding
and maximally clustered, Deodhar's algorithm \cite{d} gives a simple
combinatorial formula for the Kazhdan--Lusztig polynomials associated to such
permutations.  This also yields a combinatorial description of the corresponding
Kazhdan--Lusztig basis elements of the Hecke algebra.  The maximally clustered
permutations introduced in \cite{losonczy} are a generalization of the freely
braided permutations developed in \cite{g-l1} and \cite{g-l2}, and these in turn
include the fully commutative permutations studied in \cite{s1} as a subset.
Section~\ref{s:background} describes Deodhar's algorithm.  In
Section~\ref{s:main} we state our main result which is a formula for the
Kazhdan--Lusztig polynomials associated to maximally-clustered hexagon-avoiding
permutations.  Section~\ref{s:proof} is devoted to the proof of the formula.
Section~\ref{s:patterns} gives a pattern comparison result showing that the
hexagon-avoiding property studied in this paper can be characterized by
avoiding four 1-line patterns.  

\subsection{Background}

We view the symmetric group $S_n$ as a rank $n-1$ Coxeter group of type $A$
with the set of generators $S = \{ s_1, \dots, s_{n-1} \}$ and relations of the
form $(s_i s_{i \pm 1})^3 = 1$ together with $(s_{i}s_{j})^{2} = 1$ for $| i -
j | \geq 2$.  The Coxeter graph for $S_n$ is the graph on the generating set
$S$ with edges connecting $s_{i}$ and $s_{j}$ whenever $s_i$ does not commute
with $s_j$.  We may also refer to elements in the symmetric group by the
\textit{1-line notation} $w=[w_{1}w_{2} \dots w_{n}]$ where $w$ is the
bijection mapping $i$ to $w_{i}$.  Then the generators $s_{i}$ are the adjacent
transpositions interchanging the entries $i$ and $i+1$ in the 1-line notation.
Suppose $w=[w_1 \dots w_n]$ and $p=[p_1 \ldots p_k]$ is another permutation in
$S_{k}$ for $k\leq n$.  We say $w$ \textit{contains the permutation pattern}
$p$ or $w$ \textit{contains} $p$ \textit{as a 1-line pattern} whenever there
exists a subsequence $1\leq i_{1}<i_{2}<\ldots<i_{k}\leq n$ such that
\[ w_{i_{a}} < w_{i_{b}} \text{ if and only if } p_{a} < p_{b} \]
for all $1 \leq a < b \leq k$.  We call $(i_1, i_2, \dots, i_k)$ the \em
pattern instance\em.  For example, $[\ul{5}32\ul{4}\ul{1}]$ contains the
pattern $[321]$ in several ways, including the underlined subsequence.  If $w$
does not contain the pattern $p$, we say that $w$ \em avoids \em $p$.

An \em expression \em is any product of generators from $S$ and the \em length
\em $l(w)$ is the minimum length of any expression for the permutation $w$.
Such a minimum length expression is called \em reduced\em.  Each permutation $w
\in S_n$ can have several different reduced expressions representing it.  For
example, one reduced expression for $[3412]$ is $s_2 s_3 s_1 s_2$.  Given $w
\in S_n$, we represent reduced expressions for $w$ in sans serif font, say
$\w=\w_{1} \w_{2}\cdots \w_{p}$ where each $\w_{i} \in S$.  We call any
expression of the form $s_i s_{i \pm 1} s_i$ a \em short-braid \em after A.
Zelevinski (see \cite{fan1}).  We caution the reader that some authors have
used the term short-braid to refer to a commutation move between two entries
$s_i$ and $s_j$ where $|i - j| \geq 2$.  We say that $x < w$ \em in Bruhat
order \em if a reduced expression for $x$ appears as a subexpression that is
not necessarily consecutive, of a reduced expression for $w$.  If $s_i$ appears
as the last factor in any reduced expression for $w$, then we say that $s_i$ is
a \em descent \em for $w$; otherwise, $s_i$ is an \em ascent \em for $w$.  Let
the \em support \em of a permutation $w$, denoted $supp(w)$, be the set of all
generators appearing in any reduced expression for $w$, which is well-defined
by Tits' theorem \cite{t}.  We say that the element $w$ is \em connected \em if
$supp(w)$ is connected in the Coxeter graph of $W$.

We define an equivalence relation on the set of reduced expressions for a
permutation by saying that two reduced expressions are in the same \em
commutativity class \em if one can be obtained from the other by a sequence of \em
commuting moves \em of the form $s_i s_j \mapsto s_j s_i$ where $|i-j| \geq 2$.
If the reduced expressions for an element $w$ form a single commutativity class,
then we say $w$ is \em fully-commutative\em.

\subsection{Heaps}

If $\w = \w_1 \cdots \w_k$ is a reduced expression, then following \cite{s1} we
define a partial ordering on the indices $\{1, \cdots, k\}$ by the transitive
closure of the relation $i \lessdot j$ if $i < j$ and $\w_i$ does not commute
with $\w_j$.  We label each element $i$ of the poset by the corresponding
generator $\w_i$.  It follows quickly from the definition that if $\w$ and
$\w'$ are two reduced expressions for a permutation $w$ that are in the same
commutativity class then the labeled posets of $\w$ and $\w'$ are isomorphic.
This isomorphism class of labeled posets is called the \em heap \em of $\w$,
where $\w$ is a reduced expression representative for a commutativity class of
$w$.  In particular, if $w$ is fully-commutative then it has a single
commutativity class, and so there is a unique heap of $w$.

As in \cite{b-w}, we will represent a heap as a set of lattice points embedded
in $\N^2$.  To do so, we assign coordinates $(x,y) \in \N^2$ to each entry of
the labeled Hasse diagram for the heap of $\w$ in such a way that:
\begin{enumerate} 
\item[(1)] If an entry represented by $(x,y)$ is labeled $s_i$ in the heap, then $x = i$, and
\item[(2)] If an entry represented by $(x,y)$ is greater than an entry
represented by $(x',y')$ in the heap, then $y > y'$.
\end{enumerate}
Since the Coxeter graph of type $A$ is a path, it follows from the definition
that $(x,y)$ covers $(x',y')$ in the heap if and only if $x = x' \pm 1$, $y >
y'$, and there are no entries $(x'', y'')$ such that $x'' \in \{x, x'\}$ and $y'
< y'' < y$.  Hence, we can completely reconstruct the edges of the Hasse
diagram and the corresponding heap poset from a lattice point representation.
This representation enables us to make arguments ``by picture'' that would
otherwise be difficult to formulate.

Although there are many coordinate assignments for any particular heap, the $x$
coordinates of each entry are fixed for all of them, and the coordinate
assignments of any two entries only differ in the amount of vertical space
between them.  In the case that $w$ is fully-commutative, a canonical choice
can be made by ``coalescing'' the entries as in \cite{b-w}.  We will adhere to
this standard when we illustrate specific heaps, but our arguments should
always be viewed as referring to the underlying heap poset.  In particular,
when we consider the heaps of general permutations we will only allude to the
relative vertical positions of the entries, and never their absolute
coordinates.

\begin{example}\label{ex:heap}
One lattice point representation of the heap of $w = s_2 s_3 s_1 s_2 s_4$ is
shown below, together with the labeled Hasse diagram for the unique heap poset
of $w$.

\smallskip
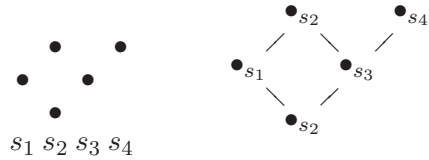
\begin{center}
\begin{tabular}{ll}
\xymatrix @=-4pt @! {
\hs & \hs & \hs & \hs & \hs \\
& \hf & \hs & \hf & \hs & \hs \\
\hf & \hs & \hf & \hs & \hs & \hs \\
& \hf & \hs & \hs & \hs & \hs \\
s_1 & s_2 & s_3 & s_4 \\
} & 
\xymatrix @=0pt @! {
 & \hf_{s_2} \ar@{-}[dl] \ar@{-}[dr] & & \hf_{s_4} \ar@{-}[dl] \\
 \hf_{s_1} \ar@{-}[dr] & & \hf_{s_3} \ar@{-}[dl] \\
 & \hf_{s_2} \\
} \\
\end{tabular}
\end{center}
\end{example}

Suppose $x$ and $y$ are a pair of entries in the heap of $\w$ that correspond
to the same generator $s_i$, so they lie in the same column $i$ of the heap.
Assume that $x$ and $y$ are a \em minimal pair \em in the sense that there is
no other entry between them in column $i$.  Then, for $\w$ to be reduced, there
must exist at least one non-commuting generator between $x$ and $y$, and if
$\w$ is short-braid avoiding, there must actually be two non-commuting labeled
heap entries that lie strictly between $x$ and $y$ in the heap.  We call these
two non-commuting labeled heap entries a \em resolution \em of the pair $x,y$.
If the generators lie in distinct columns, we call the resolution a \em distinct
resolution\em.  The Lateral Convexity Lemma of \cite{b-w} characterizes
fully-commutative permutations $w$ as those for which every minimal pair in the
heap of $w$ has a distinct resolution.

We now describe a notion of containment for heaps.  Recall from \cite{b-j1} that
an \em orientation preserving Coxeter embedding \em $f: \{s_1, \dots, s_{k-1} \}
\rightarrow \{s_1, \dots, s_{n-1} \}$ is an injective map of Coxeter generators such
that for each $m \in \{2, 3\}$, we have
\[ (s_i s_j)^{m} = 1 \text{ if and only if } (f(s_i) f(s_j))^{m} = 1 \]
and the subscript of $f(s_i)$ is less than the subscript of $f(s_j)$ whenever
$i < j$.  We can view this as a map of permutations which we also denote $f:
S_k \rightarrow S_n$ by extending it to a word homomorphism which is then
applied to any reduced expression in $S_k$.

Recall that a subposet $Q$ of $P$ is called \em convex \em if $y \in Q$
whenever $x < y < z$ in $P$ and $x, z \in Q$.  Suppose that $w$ and $h$ are
permutations.  We say that $w$ \em heap-contains \em $h$ if there exist
commutativity classes represented by $\w$ and $\h$, together with an orientation
preserving Coxeter embedding $f$ such that the heap of $f(\h)$ is contained as
a convex labeled subposet of the heap of $\w$.  If $w$ does not heap-contain
$h$, we say that $w$ \em heap-avoids \em $h$.  To illustrate, $w = s_2 s_3 s_1
s_2 s_4$ from Example~\ref{ex:heap} heap-contains $s_1 s_2 s_3$ under the
Coxeter embedding that sends $s_i \mapsto s_{i+1}$, but $w$ heap-avoids $s_1
s_2 s_1$.

In type $A$, the heap construction can be combined with another
combinatorial model for permutations in which the entries from the
1-line notation are represented by strings.  The points at which two
strings cross can be viewed as adjacent transpositions of the 1-line
notation.  Hence, we may overlay strings on top of a heap diagram to
recover the 1-line notation for the element, by drawing the strings
from bottom to top so that they cross at each entry in the
heap where they meet and bounce at each lattice point not in
the heap.  Conversely, each permutation string diagram corresponds
with a heap by taking all of the points where the strings cross as the
entries of the heap.

For example, we can overlay strings on the two heaps of $[3214]$.
Note that the labels in the picture below refer to the strings, not
the generators.
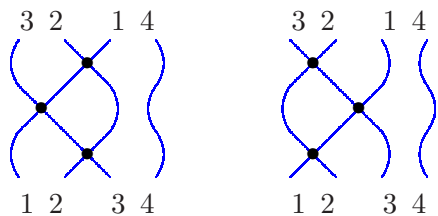
\begin{center}
\begin{tabular}{ll}
	\heap{
	& 3 \ \ 2 &  & 1 \ \ 4 & \\
	\StringR{\hs} & \hs & \StringLRX{\hf} & \hs & \StringL{\hs} \\
	\hs & \StringLRX{\hf} & \hs & \StringLR{\hs} & \hs \\
	\StringR{\hs} & \hs & \StringLRX{\hf} & \hs & \StringL{\hs} \\
	& 1 \ \ 2 &  & 3 \ \ 4 & \\
	} & 
	\heap{
	& 3 \ \ 2 &  & 1 \ \ 4 &  \\
	\hs & \StringLRX{\hf} & \hs & \StringLR{\hs} \\
	\StringR{\hs} & \hs & \StringLRX{\hf} & \hs & \StringL{\hs} \\
	\hs & \StringLRX{\hf} & \hs & \StringLR{\hs} \\
	& 1 \ \ 2 &  & 3 \ \ 4 &  \\
	} \\
\end{tabular}
\end{center}

For a more leisurely introduction to heaps and string diagrams, as well as
generalizations to Coxeter types $B$ and $D$, see \cite{b-j1}.  Cartier and
Foata \cite{cartier-foata} were among the first to study heaps of dimers, which
were generalized to other settings by Viennot \cite{viennot}.  Stembridge has
studied enumerative aspects of heaps \cite{s1,s2} in the context of fully
commutative elements.  Green has also considered heaps of pieces with
applications to Coxeter groups in \cite{green1,green2,green3}.


\subsection{Deodhar's Theorem}

Given any Coxeter group $W$, we can form the \em Hecke algebra
$\mathcal{H}$ \em over the ring $\Z[q^{1/2}, q^{-1/2}]$ with basis $\{ T_w
: w \in W \}$, and relations:
\begin{align*}
T_s T_w = & T_{s w} \text{ for } l(s w) > l(w)  \\
 (T_s)^2 = & (q - 1) T_{s} + q T_1
\end{align*}
where $T_1$ corresponds to the identity element.
Kazhdan and Lusztig \cite{k-l} described another basis for $\mathcal{H}$ that
is invariant under the involution on the Hecke algebra defined by $\overline{q}
= q^{-1}$, $\overline{T_s} = (T_s)^{-1}$, where we denote the involution with
an overline.  This basis, denoted $\{ C_w' : w \in W \}$, has important
applications in representation theory and algebraic geometry \cite{K-L2}.  The
Kazhdan--Lusztig polynomials $P_{x,w}(q)$ describe how to change between these
bases of $\mathcal{H}$:
\[ C_w' = q^{-\frac{1}{2} l(w)} \sum_{x \leq w} P_{x,w}(q) T_x.  \]
The $C_{w}'$ are defined uniquely to be the Hecke algebra elements
that are invariant under the involution and have expansion
coefficients as above, where $P_{x,w}$ is a polynomial in $q$ with
\[ \text{degree } P_{x,w}(q) \leq \frac{l(w)-l(x)-1}{2} \]
for $x<w$ in Bruhat order and $P_{w,w}(q)=1$.  We use the notation $C_{w}'$ to
be consistent with the literature because there is already a related basis
denoted $C_{w}$.

Fix a reduced expression $\w = \w_{1} \w_{2} \cdots \w_{k}$.  Define a \em
mask \em $\s$ associated to the reduced expression $\w$ to be any
binary vector $(\s_1, \cdots, \s_k)$ of length $k = l(w)$.  Every mask
corresponds with a subexpression of $\w$ defined by $\w^\s =
\w_{1}^{\s_1} \cdots \w_{k}^{\s_k}$ where
\[
\w_{j}^{\s_j}  =
\begin{cases}
\w_{j}  &  \text{ if  }\s_j=1\\
\text{id}  &  \text{ if  }\s_j=0.
\end{cases}
\]
Each $\w^\s$ is a product of generators so it determines an element of $W$.  For
$1\leq j\leq k$, we also consider initial sequences of masks denoted $\s[j] =
(\s_1, \cdots, \s_j)$, and the corresponding initial subexpressions $\w^{\s[j]}
= \w_{1}^{\s_1} \cdots \w_{j}^{\s_j}$.  In particular, we have $\w^{\s[k]} =
\w^\s$.  The mask $\s$ is \em proper \em if it does not consist of all 1
entries, since $\w^{(1, \dots, 1)} = \w$ which is the fixed reduced expression
for $w$.

We say that a position $j$ (for $2 \leq j \leq k$) of the fixed reduced
expression $\w$ is a \em defect \em with respect to the mask $\s$ if
\[ l(\w^{\s[j-1]} \w_{j}) < l(\w^{\s[j-1]}). \]
Note that the defect status of position $j$ does not depend on the value of
$\s_j$.  Let $d_{\w}(\s)$ denote the number of defects of $\w$ for a mask $\s$.
We will use the notation $d(\s) = d_{\w}(\s)$ if the reduced word $\w$ is fixed.  

Deodhar's framework gives a combinatorial interpretation
for the Kazhdan--Lusztig polynomial $P_{x,w}(q)$ as the generating function for
masks $\s$ on a reduced expression $\w$ with respect to the defect statistic
$d(\s)$.  We begin by considering subsets of 
\[ \mathcal{S} = \{ \text{ all possible masks $\s$ on $\w$ } \}. \]
For $\mathcal{E} \subset \mathcal{S}$, we define a prototype for $P_{x,w}(q)$:
\[ P_x(\mathcal{E}) = \sum_{ \substack{ \s \in \mathcal{E} \\ \w^{\s} = x } } q^{d(\s)} \]
and a corresponding prototype for the Kazhdan--Lusztig basis element $C_{w}'$:
\[ h(\mathcal{E}) = q^{-{1 \over 2} l(w)} \sum_{ \s \in \mathcal{E} } q^{d(\s)}
T_{\w^{\s}}. \]

\begin{definition}{\bf \cite{d}}\label{d:admissible}
Fix $\w = \w_1 \w_2 \dots \w_k$.  We say that $\mathcal{E} \subset \mathcal{S}$
is \em admissible on \em $\w$ if:
\begin{enumerate}
\item $\mathcal{E}$ contains $\s = (1, 1, \dots, 1)$.
\item $\mathcal{E} = \tilde{\mathcal{E}}$ where $\tilde{\s} = (\s_1, \s_2, \dots, \s_{k-1}, 1 - \s_k)$.
\item $h(\mathcal{E}) = \overline{h(\mathcal{E})}$ is invariant under the involution on the Hecke algebra.
\end{enumerate}

We say that $\mathcal{E}$ is \em bounded on \em $\w$ if $P_x(\mathcal{E})$ has
degree $\leq {1 \over 2} (l(w) - l(x) - 1)$ for all $x < w$ in Bruhat order.
\end{definition}

\begin{theorem}{\bf \cite{d}}\label{t:deodhar}
Let $x, w$ be elements in any Coxeter group $W$, and fix a reduced expression
$\w$ for $w$.  If $\mathcal{E} \subset \mathcal{S}$ is bounded and admissible
on $\w$, then 
\[ P_{x,w}(q) = P_x(\mathcal{E}) = \sum_{ \substack{ \s \in \mathcal{E} \\ \w^{\s} = x } } q^{d(\s)} \]
and hence
\[ C_{w}' = h(\mathcal{E}) = q^{-{1 \over 2} l(w)} \sum_{ \s \in \mathcal{E} } q^{d(\s)} T_{\w^{\s}}. \]
\end{theorem}

Billey and Warrington say that $w \in S_n$ is \em hexagon-avoiding \em if
it heap-avoids
\[ [46718235] = s_5 s_6 s_7 s_3 s_4 s_5 s_6 s_2 s_3 s_4 s_5 s_1 s_2 s_3. \]
When $w$ is fully-commutative, this condition is equivalent to avoiding $[46718235]$,
$[46781235]$, $[56718234]$ and $[56781234]$ as permutation patterns.  We will
show in Section~\ref{s:patterns} that this permutation pattern characterization
remains true in more general settings.

\begin{theorem}{\bf \cite{b-w}}\label{t:b-w}
The set $\mathcal{S}$ is bounded and admissible on a reduced expression $\w$
if and only if the corresponding permutation $w$ is $[321]$-avoiding and
hexagon-avoiding.
\end{theorem}

More generally, let $W$ be any Coxeter group and $\mathcal{E} \subset
\mathcal{S}$ be a set of masks on some reduced expression $\w \in W$.  By Lemma
2 of \cite{b-w}, we have that $\mathcal{E}$ is bounded if and only if for every
proper mask $\s \in \mathcal{E} \setminus \{ (1, 1, \dots, 1) \}$, we have
\begin{equation}\label{e:deodhar.ineq}
\text{\# of zero-defects of } \s < \text{\# of plain-zeros of } \s,
\end{equation}
which we refer to as the \em Deodhar bound\em.  Here, a position in $\w$ is a
\em zero-defect \em if it has mask-value 0 and it is also a defect.  A position
in $\w$ is a \em plain-zero \em if it has mask-value 0 and it is not a defect.
We say that an element represented by $\w$ is \em Deodhar \em if $\mathcal{S}$
is bounded on $\w$.

\subsection{Maximally clustered elements}

In \cite{losonczy}, Losonczy introduced the maximally clustered elements of
simply laced Coxeter groups.  We will define a set of masks for the
maximally-clustered hexagon-avoiding permutations that generalizes
Theorem~\ref{t:b-w}.

\begin{definition}{\bf \cite{losonczy}}\label{d:mc}
A \em braid cluster \em is an expression of the form \[ s_{i_1} s_{i_2} \dots
s_{i_k} s_{i_{k+1}} s_{i_k} \dots s_{i_2} s_{i_1} \] where each $s_{i_p}$ for
$1 \leq p \leq k$ has a unique $s_{i_q}$ with $p < q \leq k+1$ such that
$|i_p - i_q| = 1$.  

Let $w$ be a permutation and let $N(w)$ denote the number of $[321]$ pattern
instances in $w$.  We say $w$ is \em maximally-clustered \em if there is a
reduced expression for $w$ of the form 
\[ a_0 c_1 a_1 c_2 a_2 \dots c_M a_M \]
where each $a_i$ is a reduced expression, each $c_i$ is a braid cluster with
length $2 n_i + 1$ and $N(w) = \sum_{i=1}^{M} n_i$.  Such an expression is
called \em contracted\em.  In particular, $w$ is \em freely-braided \em if there
is a reduced expression for $w$ with $N(w)$ disjoint short-braids.  
\end{definition}
Note that this is not the original definition for the maximally clustered
elements; however it is equivalent.  The remarks in Section 5 of \cite{g-l1}
show that the number of $[321]$ pattern instances in $w$ equals the number of
contractible triples of roots in the inversion set of $w$.  Corollary 4.3.3 (ii)
and Corollary 4.3.5 of \cite{losonczy} prove that $\w$ is a contracted reduced
expression for a maximally-clustered element if and only if it has the form
given in Definition~\ref{d:mc}.  Observe that a maximally-clustered permutation
$w$ is fully-commutative if and only if $N(w) = 0$ by \cite{b-j-s}.

In type $A$, there exists a standard form for the braid clusters.

\begin{lemma}\label{l:braid_cluster}
Suppose $x = s_{i_1} s_{i_2} \dots s_{i_{k}} s_{i_{k+1}} s_{i_{k}} \dots s_{i_2}
s_{i_1}$ is a braid cluster of length $2 k + 1$ in type $A$.  Then, $x = s_{m+1}
s_{m+2} \dots s_{m+{k}} s_{m+k+1} s_{m+{k}} \dots s_{m+2} s_{m+1}$ for some $m$.
\end{lemma}
\begin{proof}
By Lemma 4.1.3 of \cite{losonczy}, there exists a sequence of moves that result
in a braid cluster for the same element $x$ such that the largest generator
$s_{m+k+1}$ appearing in any reduced expression for $x$ appears in the middle
position.

The set of generators appearing in reduced expressions for $x$ must consist of a
connected path in the Coxeter graph, otherwise the original expression for $x$
is not reduced.  By the uniqueness statement in Definition~\ref{d:mc}, the entry
next to $s_{m+k+1-i}$ must be $s_{m+k-i}$ for each $i = 0, \dots, k-1$.
\end{proof}

For our work, we will implicitly assume that any braid cluster has the
canonical form of Lemma~\ref{l:braid_cluster}.  Also, we refer to $s_{m+k}
s_{m+k+1} s_{m+k}$ as the \em central braid \em of the braid cluster 
\[ s_{m+1} s_{m+2} \dots s_{m+k} s_{m+k+1} s_{m+k} \dots s_{m+2} s_{m+1}. \]

Recall that the maximally-clustered permutations are characterized by
avoiding the permutation patterns 
\begin{align}\label{e:mc.patterns}
	\text{ $[3421]$, $[4312]$, and $[4321]$ }
\end{align}
as a result of Proposition 3.2.1 in \cite{losonczy},
while the freely-braided permutations are characterized by avoiding
\begin{align}\label{e:fb.patterns}
	\text{ $[4231]$, $[3421]$, $[4312]$, and $[4321]$ }
\end{align}
as permutation patterns by Proposition 5.1.1 in \cite{g-l1}.


\bigskip
\section{Main Result}\label{s:main}

Given a contracted expression $\w$ for a maximally-clustered hexagon-avoiding
permutation, our main result is that we can identify a set of masks on $\w$
that turn out to be bounded and admissible.  Moreover, this set has a simple
non-recursive description.

\begin{definition}
Let $\w$ be a contracted expression for a maximally-clustered hexagon-avoiding
permutation, where each braid cluster has the form given in
Lemma~\ref{l:braid_cluster}.  We say that a mask $\s$ on $\w$ has a \em {\tt
10*}-instance \em if it has the values 
\[ \pblock{\dots & s_i & s_{i+1} & s_i & \dots \\ * & 1 & 0 & * & *} \]
on any central braid instance $s_i s_{i+1} s_i$ of any braid cluster in
$\w$, where $*$ denotes an arbitrary mask value.  If $\s$ never has the values 1
and 0 (respectively) on the first two entries in any central braid of
$\w$, then we say that $\s$ is a \em {\tt 10*}-avoiding \em mask for $\w$.
\end{definition}

\begin{theorem}\label{t:mc_main}
Let $\w$ be a contracted expression for a maximally-clustered hexagon-avoiding
permutation in $S_n$, and let $\E_{\w}$ be the set of {\tt 10*}-avoiding masks
on $\w$.  Then for any $x \in S_n$,
\[ P_{x,w}(q) = P_x(\E_{\w}) = \sum_{ \substack{ \s \in \E_{\w} \\ \w^{\s} = x } } q^{d(\s)} \]
and hence
\[ C_{w}' = h(\E_{\w}) = q^{-{1 \over 2} l(w)} \sum_{ \s \in \E_{\w} } q^{d(\s)} T_{\w^{\s}} \]
\end{theorem}
\begin{proof}
This follows from Theorem~\ref{t:deodhar} since $\E_{\w}$ is bounded and
admissible by Propositions~\ref{p:bounded} and \ref{p:admissible} below.
\end{proof}

We begin by showing that the contracted expressions for maximally-clustered
permutations have an especially nice form.

\begin{lemma}\label{l:unique}
Let $\w$ be a contracted reduced expression for a maximally clustered
permutation, so $\w$ has the form
\[ a_0 c_1 a_1 \dots c_{M} a_{M} \]
where each $c_j$ is a braid cluster, and the $a_j$ are short-braid
avoiding.  Then, any generator $s_i$ that appears in any of the braid clusters
$c_j$ does not appear anywhere else in $\w$.
\end{lemma}
\begin{proof}
Suppose for the sake of contradiction that there exists a contracted reduced
expression of the form 
\[ \dots s_{m+1} s_{m+2} \dots s_{m+{k-1}} s_{m+k} s_{m+{k-1}} \dots s_{m+2} s_{m+1} \dots s_i \dots \]
where $m+1 \leq i \leq m+k$.  Then, we may choose $s_i$ to be leftmost to obtain
a contracted reduced expression of the form 
\[ \tilde{\w} = s_{m+1} s_{m+2} \dots s_{m+{k-1}} s_{m+k} s_{m+{k-1}} \dots s_{m+2} s_{m+1} \dots s_i \]
in which there are no $s_j$ generators among the entries between the braid
cluster and $s_i$ for any $m+1 \leq j \leq m+k$.  Since this is a factor of a
contracted expression, $\tilde{\w}$ is maximally-clustered by
Definition~\ref{d:mc}.

We examine the 2-line notation that is built up from the identity permutation
by multiplying on the right by $\tilde{\w}$.  The columns of this notation
encode $\block{j \\ w_j}$ for each $j \in \{1, \dots n\}$.  In particular, the
lower row is the usual 1-line notation.  The initial braid cluster produces a
consecutive $[(k+1) 2 3 \dots k 1]$-instance 
\[ \block{ \dots & m+1 & m+2 & m+3 & \dots & m+k & m+k+1 &  \dots \\ \dots & m+k+1 & m+2 & m+3 & \dots & m+k & m+1 & \dots } \]
in positions $(m+1) \dots (m+k+1)$.  

Since there are no $s_j$ among the generators not explicitly shown for any $m+1
\leq j \leq m+k$, the entries in positions $m+2, \dots, m+k$ remain fixed as we
multiply on the right by subsequent entries from $\tilde{\w}$.  Therefore, if
$m+1 < i < m+k$ then $s_i$ creates a descent among the $m+2, \dots, m+k$ entries
so the 1-line notation for $\tilde{\w}$ contains a $[4321]$ instance,
contradicting that $\tilde{\w}$ is maximally-clustered.

Next, suppose that $i = m+k$.  Since there are no $s_{m+k}$ generators between
the braid cluster and $s_i$ in $\tilde{\w}$, the entry with value $m+k$ remains
strictly left of position $m+k+1$, and all of the entries except $m+1$ that lie
strictly right of position $m+k+1$ always have values $> m+k+1$.  Since the
1-line entries in positions $m+k$ and $m+k+1$ are inverted, we cannot apply
$s_{m+k}$ again in a reduced fashion until we first apply an $s_{m+k+1}$.  

Let $x$ be the value of the entry in position $m+k+1$ just after the last
$s_{m+k+1}$ occurs, so $x > m+k+1$ and we have
\[ \block{ \dots & p & \dots & m+k & m+k+1 & \dots & q & \dots \\ 
\dots & (m+k+1) & \dots & m+k & x & \dots & (m+1) & \dots } . \]
Once we apply the last $s_{m+k}$, we obtain a $[3421]$ pattern instance,
contradicting the maximally-clustered hypothesis.  

A similar argument shows that if $i = m+1$ then the 1-line notation for
$\tilde{\w}$ contains $[4312]$ as a permutation pattern, contradicting that
$\tilde{\w}$ is maximally-clustered.  Hence, the generators of every braid
cluster appear uniquely in any contracted reduced expression for a
maximally-clustered permutation.
\end{proof}

\begin{remark}
Although we have emphasized the type $A$ case, there is a definition of
maximally-clustered and freely-braided for all simply-laced Coxeter groups.  It
is not true, even in type $D$, that these conditions imply uniqueness for the
generators in the short-braid instances.  For example, the expression $\w = s_2
s_3 s_2 s_1 s_{\tilde{1}} s_2$ in $D_4$ is contracted and freely-braided, but it
contains an $s_2$ generator beyond the short-braid instance $s_2 s_3 s_2$.
Here, we have labeled the generators so that $s_2$ is adjacent to $s_1$,
$s_{\tilde{1}}$ and $s_3$ in the Coxeter graph.
\end{remark}

In type $A$, there is an algorithm for producing a contracted expression from
the 1-line notation of the permutation, which is useful for creating computer
programs.  We have included a description of this algorithm in
Appendix~\ref{s:appendix}.  This algorithm shows that the number of braid
clusters in a maximally clustered permutation $w$ is precisely the number of
$[(m+2) 2 3 \dots (m+1) 1]$-instances in the 1-line notation for $w$.  Each
such pattern instance contributes $m$ to $N(w)$.


\bigskip
\section{Proof of the main theorem}\label{s:proof}

In this section we will prove Theorem~\ref{t:mc_main} by showing that the set of
{\tt 10*}-avoiding masks is bounded and admissible.  The proof of each of these
properties relies on a map of contracted reduced expression/mask pairs.  Here we
show that the hexagon-avoiding property is preserved under such maps, which
will be used to make inductive arguments in the proofs of
Propositions~\ref{p:bounded} and \ref{p:admissible}.

\begin{lemma}\label{l:hex.induction}
Let $\w$ be a contracted reduced expression for a maximally-clustered
hexagon-avoiding permutation, so $\w$ has the form
\[ \w = a_0 c_1 a_1 \dots a_{M-1} c_{M} a_{M} \]
where each $c_j$ is a braid cluster, and the $a_j$ are short-braid avoiding.
Let $\u$ be any expression obtained from $\w$ by removing some of the entries
from the last braid cluster $c_{M}$ in such a way that 
\[ \u = a_0 c_1 a_1 \dots a_{M-1} \widetilde{c_M} a_M \]
is still reduced and contracted.  Then, the corresponding element $u$ is
hexagon-avoiding.
\end{lemma}
\begin{proof}
We suppose that $u$ contains a hexagon pattern and show that $w$ must also
contain a hexagon pattern.  If $u$ contains a hexagon then there is some
reduced expression $\tilde{\u}$ for $u$ such that the heap of $\tilde{\u}$
contains a hexagonal subheap
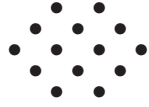
\[ \xymatrix @=-4pt @! {
& & {\hf} & & {\hf} & & {\hs} & \\
& {\hf} & & {\hf} & & {\hf} & & {\hs} \\
{\hf} & & {\hf} & & {\hf} & & {\hf} & \\
& {\hf} & & {\hf} & & {\hf} & & {\hs} \\
& & {\hf} & & {\hf} & & {\hs} & 
} \]
in columns $\{ 1+i, \dots, 7+i \}$ for some $i \geq 0$.  

Note that although we assume $\u$ is a contracted reduced
expression, our notation $a_0 c_1 a_1 \dots a_{M-1} \widetilde{c_M} a_M$ might
not represent the partition of this expression into braid clusters, because the
factors in this notation were defined with respect to $\w$.  In particular,
$a_{M-1} \widetilde{c_M} a_M$ may be fully-commutative.  In any case, every
other braid cluster $c_i$ of $\w$ where $1 \leq i \leq M-1$ remains a braid
cluster in $\u$, and we assume these are in the canonical form of
Lemma~\ref{l:braid_cluster}.  Also, since $u$ is maximally-clustered, the heaps
of $\u$ and $\tilde{\u}$ differ only in the choice of commutativity classes for
their braid clusters by \cite[Corollary 4.3.3]{losonczy}.  

Let $[p,q]$ be the interval of columns that support the braid cluster $c_M$ in
the heap of $\w$.  By Lemma~\ref{l:unique}, we have that the heaps of $\w$ and
$\u$ agree on the columns outside of $[p,q]$.  In particular, if the hexagon
instance in $\tilde{\u}$ does not use any entries from columns $[p,q]$, then
the hexagon instance appears in the reduced expression for $w$ that is obtained
by choosing the commutativity class of each braid cluster $c_1, \dots, c_{M-1}$
to match the commutativity class of each such braid cluster in $\tilde{\u}$.
As this contradicts the hypothesis that $w$ is hexagon-avoiding, we have that
the hexagon instance uses some entry of $\tilde{\u}$ from columns $[p,q]$.

Since the hexagon is fully-commutative, every minimal pair of entries in each
column has a distinct resolution.  However, by the uniqueness in
Definition~\ref{d:mc}, no minimal pair of entries in any braid cluster has a
distinct resolution.  Thus, we find that the hexagon instance in $\tilde{\u}$
either includes an entry from column $p$ that corresponds to $s_{i+7}$, or it
includes an entry from column $q$ that corresponds to $s_{i+1}$.  Let us suppose
that we have the former case without loss of generality.  Since we only remove
entries from columns $[p,q]$ as we pass from $\w$ to $\tilde{\u}$, we find that
the entries of $\w$ in column $p-1$ that are used in the hexagon instance of
$\tilde{\u}$ surround the entries from the braid cluster in column $p$.

Hence, if we choose the commutativity class for the braid cluster $c_{M}$ which
has the form
\[ s_{q} s_{q-1} \dots s_{p+1} s_{p} s_{p+1} \dots s_{q-1} s_{q} \]
then the corresponding heap contains a single entry in column $p$, and so we
obtain a hexagon instance in this heap for $w$.  This contradicts our hypothesis
that $w$ is hexagon-avoiding.  An illustration for the case when the hexagon
instance uses $s_{7+i}$ from $c_{M}$ is given below.  The ${\color{blue} \hh}$
points are entries from the braid cluster $c_M$ in the heap of $\w$.
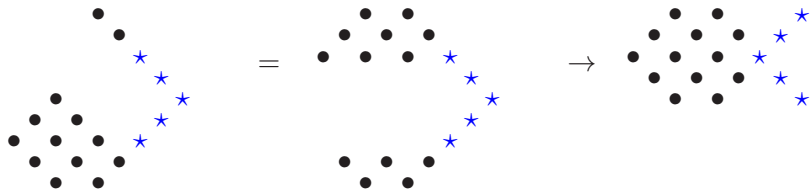
\[
\xymatrix @=-4pt @! {
& & {\hs} & & {\hf} & & {\hs} & \\
& {\hs} & & {\hs} & & {\hf} & & {\hs} \\
{\hs} & & {\hs} & & {\hs} & & {\color{blue} \hh} & \\
& {\hs} & & {\hs} & & {\hs} & & {\color{blue} \hh} \\
{\hs} & & {\hf} & & {\hs} & & {\hs} & & {\color{blue} \hh} & \\
& {\hf} & & {\hf} & & {\hs} & & {\color{blue} \hh} \\
{\hf} & & {\hf} & & {\hf} & & {\color{blue} \hh} & \\
& {\hf} & & {\hf} & & {\hf} & & {\hs} \\
& & {\hf} & & {\hf} & & {\hs} & \\
} 
\parbox[t]{0.5in}{ \vspace{0.2in} \hspace{0.2in} $ = $ }
\xymatrix @=-4pt @! {
& & {\hf} & & {\hf} & & {\hs} & \\
& {\hf} & & {\hf} & & {\hf} & & {\hs} \\
{\hf} & & {\hf} & & {\hf} & & {\color{blue} \hh} & \\
& {\hs} & & {\hs} & & {\hs} & & {\color{blue} \hh} \\
{\hs} & & {\hs} & & {\hs} & & {\hs} & & {\color{blue} \hh} & \\
& {\hs} & & {\hs} & & {\hs} & & {\color{blue} \hh} \\
& & {\hs} & & {\hs} & & {\color{blue} \hh} & \\
& {\hf} & & {\hf} & & {\hf} & & {\hs} \\
& & {\hf} & & {\hf} & & {\hs} & \\
} 
\parbox[t]{0.5in}{ \vspace{0.2in} \hspace{0.2in} $\rightarrow$ }
\xymatrix @=-4pt @! {
& & {\hf} & & {\hf} & & {\hs} & & {\color{blue} \hh} \\
& {\hf} & & {\hf} & & {\hf} & & {\color{blue} \hh} \\
{\hf} & & {\hf} & & {\hf} & & {\color{blue} \hh} & \\
& {\hf} & & {\hf} & & {\hf} & & {\color{blue} \hh} \\
& & {\hf} & & {\hf} & & {\hs} & & {\color{blue} \hh} \\
}
\]
\end{proof}

In the proof of the next result, we \em decorate \em the heap diagrams according
to mask-value using Table~\ref{ta:1}.

\bigskip
\begin{center}
\begin{table}[ht]
\centering
\caption{Heap decorations}
\begin{tabular}{|p{0.7in}|p{3in}|} 
\hline
Decoration & Mask-value \\
\hline
$\hd$  &	zero-defect entry\\
$\hz$  &	plain-zero entry (not a defect)\\
$\hf$  &	mask-value 1 entry\\
\hline
\end{tabular}
\label{ta:1}
\end{table}
\end{center}
\bigskip 

\begin{proposition}\label{p:bounded}
Let $\w$ be a contracted expression for a maximally-clustered hexagon-avoiding
permutation $w$.  Then, the set of {\tt 10*}-avoiding masks on $\w$ is bounded.
\end{proposition}
\begin{proof}
Suppose that there exists some proper {\tt 10*}-avoiding mask $\s$ on $\w$ with
\[ \text{ \# zero-defects of } \s \geq \text{ \# plain-zeros of } \s \]
which violates Equation~(\ref{e:deodhar.ineq}).  In particular, $\s$ contains at
least one zero-defect because it is a proper mask.  We show how to extend this
mask to an element with one fewer braid cluster while maintaining the
non-Deodhar bound.  Eventually, we derive a contradiction by
Theorem~\ref{t:b-w}.

We may adorn the heap diagram of the permutation with strings that correspond to
entries in the 1-line notation for the permutation.  This construction is a
standard technique which is given a detailed description in \cite{b-j1}.  We can
consider a pair of strings emanating from each entry of the decorated heap such
that the strings cross at mask-value 1 entries and bounce at mask-value 0
entries.  It follows from the definition that a defect entry must have a pair of
strings that cross an odd number of times below the defect.  Suppose the last
braid cluster has length $2 k + 1$ with entries occupying columns $m, m+1,
\dots, m+k$ of the heap, and it is in the form given by
Lemma~\ref{l:braid_cluster}.  Let $\mathcal{B}(\w)$ denote the reduced
expression obtained from $\w$ by removing the last $k$ entries of this last
braid cluster.  By Lemma~\ref{l:unique} and Definition~\ref{d:mc}, we find that
$\mathcal{B}(\w)$ is a contracted reduced expression for a maximally clustered
permutation with one fewer braid cluster.  We describe how to construct a
non-Deodhar mask on $\mathcal{B}(\w)$, starting from the restriction of $\s$ to
$\mathcal{B}(\w)$.

First, observe that we remove at least as many plain-zeros as zero-defects from
columns $m, m+1, \dots, m+k-1$ when we apply $\mathcal{B}$.  To see this,
suppose there exists a zero-defect at the top of column $h$ where $m \leq h \leq
m+k-1$.  Then the strings for the defect must cross below the defect.  By
Lemma~\ref{l:unique} and \ref{l:braid_cluster}, the form of the heap in columns
$m, \dots, m+k$ is determined as shown in Figure~\ref{f:bounded_heaps}(a).  In
particular, there can be no entries in column $m-1$ lying between the two
entries in column $m$ by Definition~\ref{d:mc}.

\begin{figure}[h]
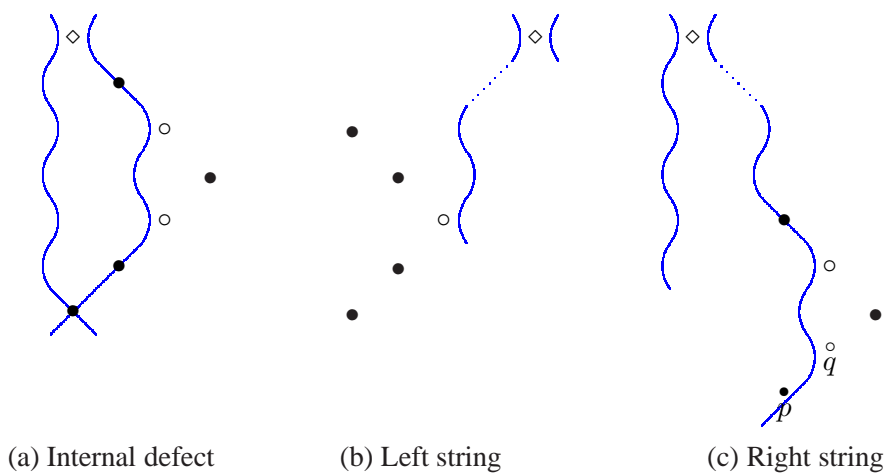

\begin{center}
\begin{tabular}{ccc}
\heap {
& \StringLR{\hd} & & {\hs}  \\
\StringR{\hs} & & \StringLX{\hf}  \\
& \StringL{\hs} & & \StringL{\hz}  \\
\StringR{\hs} & & \StringR{\hs} & & {\hf} \\
& \StringL{\hs} & & \StringL{\hz}  \\
\StringR{\hs} & & \StringRX{\hf}  \\
& \StringLRX{\hf} & & {\hs}  \\
} &
\heap {
& & {\hs} & & {\hs} & & \StringLR{\hd} & \\
& {\hs} & & {\hs} & & \StringRXD{\hs} & & {\hs} \\
& & {\hf} & & \StringR{\hs} & & {\hs} & \\
& {\hs} & & {\hf} & & \StringL{\hs} & & {\hs} \\
{\hs} & & {\hs} & & \StringR{\hz} & & {\hs} & \\
& {\hs} & & {\hf} & & {\hs} & & {\hs} \\
{\hs} & & {\hf} & & {\hs} & & {\hs} & \\
} &
\heap {
& \StringLR{\hd} & & {\hs} & & {\hs} & & {\hs} \\
\StringR{\hs} & & \StringLXD{\hs} & & {\hs} & & {\hs} & \\
& \StringL{\hs} & & \StringL{\hs} & & {\hs} & & {\hs} \\
\StringR{\hs} & & \StringR{\hs} & & {\hs} & & {\hs} & \\
& \StringL{\hs} & & \StringLX{\hf} & & {\hs} & & {\hs} \\
\StringR{\hs} & & {\hs} & & \StringL{\hz} & & {\hs} & \\
& {\hs} & & \StringR{\hs} & & {\hf} & & {\hs} \\
{\hs} & & {\hs} & & \StringL{\stackrel{\hz}{q}} & & {\hs} & \\
& {\hs} & & \StringRX{\stackrel{\hf}{p}} & & {\hs} & & {\hs} \\
} \\
(a) Internal defect &
(b) Left string &
(c) Right string \\
\end{tabular}
\end{center}
\caption{{\tt 10*}-avoiding masks are bounded}
\label{f:bounded_heaps}
\end{figure}

Since the mask is {\tt 10*}-avoiding, the right string of the defect must travel
southeast from the defect until it hits a zero in column $g$ with $h < g <
m+k$, drop straight down until it hits the next entry in the heap which must
also be a zero in the same column $g$, and then continue southwest until it
crosses the left string of the defect at the bottom entry in column $h$.
Hence, both of the entries in column $g$ must have mask-value 0 to facilitate
the string crossing for the defect.  Neither entry in column $g$ can be a
zero-defect, because such a defect must have a mask-value 1 entry directly below
it in the same column to facilitate the string crossing.  As we already assumed
that the mask values of both entries in column $g$ were 0, the lower entry is
not a critical generator.  Thus, removing the top entries in columns $m, m+1,
\dots m+k-1$ removes a plain-zero for every defect, so the non-Deodhar bound for
the mask $\s$ restricted to $\mathcal{B}(\w)$ is preserved.

Next, we must consider whether the removal of the last $k$ entries from the last
braid cluster might destroy the defect status of an entry further to the right
in the contracted reduced expression.  We will argue that it cannot.  Any
zero-defect whose left string intersects the braid cluster will have the same
string dynamics after we apply $\mathcal{B}$ as it did originally since we do
not change the mask-values of the remaining entries in the braid cluster.  This
is illustrated in Figure~\ref{f:bounded_heaps}(b).  Hence, applying
$\mathcal{B}$ does not destroy the defect status of such a zero-defect.

Suppose the right string of a zero-defect intersects the braid cluster as in
Figure~\ref{f:bounded_heaps}(c).  Then, the path of this string is prescribed as
in the argument above.  It must intersect a zero in column $g$ with $g < m+k$,
drop straight down until it hits the next entry in the heap which must also be a
zero in the same column $g$, and then continue southwest until it leaves the
braid cluster and eventually crosses the left string of the zero-defect.  No
other mask configuration will allow the strings of the zero-defect to cross by
Lemma~\ref{l:unique}.

Observe that since there are no entries in column $m-1$ between the two entries
in column $m$, the strings cannot cross at $p$.  If the mask-value of $p$ is 1,
then we change the mask for $\mathcal{B}(\w)$ by interchanging the mask-values
of $p$ and $q$.  Then the string dynamics for the defect will remain the same
after we apply $\mathcal{B}$, and we preserve the non-Deodhar bound of the mask.

After iteratively applying $\mathcal{B}$ to remove each braid cluster of $\w$,
we eventually obtain a short-braid avoiding reduced expression that is still
hexagon-avoiding by Lemma~\ref{l:hex.induction}.  The argument above shows that
we have a proper mask on the resulting expression in which no more defects than
plain-zeros have been removed in comparison with the original mask on $\w$.
But Theorem~\ref{t:b-w} implies that every proper mask on an element that is
short-braid avoiding and hexagon-avoiding must satisfy the Deodhar bound from
Equation~(\ref{e:deodhar.ineq})
\[ \text{ \# zero-defects of } \s < \text{ \# plain-zeros of } \s . \]
This contradicts the existence of our original non-Deodhar mask on $\w$.
\end{proof}

\begin{proposition}\label{p:admissible}
Let $\w$ be a contracted expression for a maximally-clustered hexagon-avoiding
permutation.  Then, the set of {\tt 10*}-avoiding masks on $\w$ is admissible.
\end{proposition}
\begin{proof}
We let $\E_{\w}$ denote the set of {\tt 10*}-avoiding masks on $\w$.  In order
for $\E_{\w}$ to be admissible, we must show that it satisfies the three
properties in Definition~\ref{d:admissible}.

The first property follows since the mask $\s = (1, 1, \dots, 1)$ avoids {\tt
10*}, and the second property holds because avoiding {\tt 10*} imposes no
restrictions on the last entry of a mask.

In order to show that $h(\E_{\w})$ is invariant under the Hecke algebra
involution, we will use that
\[ h(\E_{\w}) = \overline{ h(\E_{\w}) } \text{ if and only if } h(\E_{\w}^c) = \overline{ h(\E_{\w}^c) } \]
where $\E_{\w}^c$ denotes the set complement $\mathcal{S} \setminus \E_{\w}$.
Indeed, Deodhar observes in \cite[Proposition 3.5]{d} that $h(\mathcal{S}) =
C_{\w_1}' C_{\w_2}' \dots C_{\w_{l(\w)}}'$, so in particular $h(\mathcal{S}) =
\overline{ h(\mathcal{S}) }$.  Since $\mathcal{S} = \E_{\w} \sqcup \E_{\w}^c$,
we have $h(\E_{\w}) = h(\mathcal{S}) - h(\E_{\w}^c)$, so by linearity any set
of masks is invariant under the involution whenever its complement is.

We proceed by induction on the number $N(w)$ of short-braids in the contracted
expression $\w$.  If $\w$ is short-braid avoiding, then $\E_{\w} =
\mathcal{S}$, and this set is admissible.

Next, suppose that $\w$ has $N(\w)$ short-braids, and that $h(\E_{\v}) =
\overline{ h(\E_{\v}) }$ whenever $\v$ is a contracted expression for a
maximally-clustered hexagon-avoiding permutation with fewer than $N(\w)$
short-braids.  We will prove that $h(\E_{\w}^c)$ is invariant under the
involution.

Suppose $B \subset \{1, \dots, l(\w)\}$ is a set of positions in $\w$.  Let
$\E_{\w}^{B}$ denote the set of masks $\s$ on $\w$ such that $\s$ has a {\tt
10*} instance starting at position $b$ if and only if $b \in B$.  Not every
subset can correspond to a set of masks; for example, if $B$ includes positions
that do not correspond to central braids in braid clusters of $\w$ then
$\E_{\w}^{B}$ will necessarily be the empty set.  In such a situation we define
$h(\E_{\w}^{B}) = 0$.

Observe that we can decompose $\E_{\w}^c$ as $\sqcup \E_{\w}^{B}$, where we
take the disjoint union over all non-empty sets of positions.  Then, we have
\[ h(\E_{\w}^c) = \sum_{\text{non-empty position sets } B } h( \E_{\w}^{B} ) \]
so it suffices by linearity to show that
\[ h( \E_{\w}^{B} ) = \overline{ h( \E_{\w}^{B} ) } \]
for each non-empty set of positions $B$.

Towards this end, we define a map $\f$ on reduced expression/mask pairs $(\w,
\s)$.  We denote the image of $\w$ under the map by $\f(\w)$ via an abuse of
notation since $\f(\w)$ depends on the mask $\s$.  Then, 
\[ \f : \{ \w \} \times \E_{\w}^{B} {\longrightarrow} \{ \f(\w) \} \times
\E_{\f(\w)}^{B \setminus \mathbf{max}(B)} \] 
is defined by:
\begin{align*}
\pblock{
\dots & s_{i} & \dots & s_{i+k-1} & s_{i+k} & s_{i+k-1} & \dots & s_{i+1} & s_{i} & \dots \\ 
\dots & 1 & \dots & 1 & 0 & 1 & \dots & 1 & * & \dots \\ 
& \scriptstyle{ b_1-k+1 } & \dots & \scriptstyle{ b_1 } & \scriptstyle{ b_1 + 1
} & \scriptstyle{ b_1 + 2 } & \dots & \scriptstyle{ b_1+k } & \scriptstyle{
b_1+k+1 } \\ 
} 
\stackrel{\f}{\mapsto} 
\pblock{
\dots & s_{i} & \dots \\ 
\dots & (1-*) & \dots \\  
\dots & \scriptstyle{ b_1-k+1 } & \dots \\
}
\end{align*}
where the notation indicates the mask-value and location of each entry on the
second and third rows of the array, respectively.  The segment between indices
$b_1-k+1$ and $b_1+k+1$ is chosen to be the maximal segment that has symmetric
intervals of mask-value 1 entries about the central braid.  In particular, the
segment may not be the entire braid-cluster instance in $\w$.  In the
definition, $b_1 = \mathbf{max}(B)$ is the rightmost position in $B$, and
$(1-*)$ indicates that we take the mask-value on $s_i$ that is opposite to the
mask-value given on the second $s_i$ generator in location $b_1+k+1$ of the {\tt
10*} instance.  In the image, the entire segment between indices $b_1-k+1$ and
$b_1+k+1$ is replaced by the single entry at location $b_1-k+1$.

Observe that applying this move $\f$ to the reduced expression/mask pairs
$(\w,\s) \in \{ \w \} \times \E_{\w}^{B}$ has the following effects:
\begin{enumerate}
	\item[(1)] The map $\f$ gives a bijection between $\E_{\w}^{B}$ and
		$\E_{\f(\w)}^{B \setminus \mathbf{max}(B)}$ because it is
		reversible.
	\item[(2)] $\f(\w)$ is a contracted reduced expression for a
		maximally-clustered hexagon-avoiding permutation by
		Lemma~\ref{l:hex.induction} and Lemma~\ref{l:unique}, with
		exactly $k$ fewer short-braid instances than $\w$, and
		$l(\f(\w)) = l(\w) - 2 k$.
	\item[(3)] $\f$ removes exactly $k$ defects from the mask $\s$ on $\w$,
		because $w^{\s[b_1+k+1]} = \f(\w)^{\f(\s)[b_1-k+1]}$, so the defect
		status of subsequent entries remains precisely the same.
	\item[(4)] $w^{\s} = \f(\w)^{\f(\s)}$ follows from (3).
	\item[(5)] The map $\f$ introduces no new {\tt 10*}-instances because
		the choice of the segment $(b_1-k+1, b_1+k+1)$ about $b_1$ is maximal.
\end{enumerate}

To verify (3), note that by Lemma~\ref{l:unique} there are no other $s_j$
generators for $i \leq j \leq i+k$ anywhere else in the contracted expression
$\w$.  Hence, the second $s_{i+k-1}, \dots, s_i$ generators are always defects,
while none of the entries in the first segment are.  Since $s_{i+k}$ has
mask-value 0, removing this entry has no effect on the defect status of any
subsequent entries in $\w$.  Furthermore, we can also remove the other entries
and update the mask as shown without changing the defect status of any
subsequent entry.  This follows because if both $s_i$ entries had mask-value 1,
then they would cancel after the collapse of the pairs of generators $s_{i+1},
\dots, s_{i+k-1}$ in any calculation of the defect status of an entry further to
the right.  If the second $s_i$ had mask-value 0, then it would play no role in
the calculation of the defect status of an entry further to the right, so could
be removed.

Next, we may calculate
\begin{align*}
h(\E_{\w}^{B}) 
&= q^{-{1 \over 2} l(w)} \sum_{ \s \in \pblock{ \dots & s_{i} & \dots & s_{i+k-1} & s_{i+k} & s_{i+k-1} & \dots & s_{i+1} & s_{i} & \dots \\ \dots & 1 & \dots & 1 & 0 & 1 & \dots & 1 & * & \dots \\ }} q^{d(\s)} T_{\w^{\s}} \\ 
\ \\
&= q^{-{1 \over 2} (l(w))} \sum_{ \s \in \pblock{\dots & s_i & \dots \\ * &
(1-*) & * \\
} } q^{(d(\s)+k)} T_{\f(\w)^{\s}} \\
&= q^{-{1 \over 2} (l(w)-2 k)} \sum_{ \s \in \pblock{\dots & s_i & \dots \\ * &
(1-*) & * \\
} } q^{d(\s)} T_{\f(\w)^{\s}} \\
\end{align*}
to see that 
\[ h( \E_{\w}^{B} ) = h( \E_{\f(\w)}^{B \setminus \mathbf{max}(B)} ). \]

If we let $m = |B|$ then the masks in $\E_{\w}^{B}$ all have exactly $m$ {\tt
10*}-instances.  Therefore, after $m$ applications of the map $\f$ we will
obtain a set $\f^m( \E_{\w}^{B} )$ of {\tt 10*}-avoiding masks on some
contracted expression $\f^m(\w)$ that has \em at least $m$ fewer short-braids than
$\w$\em.  This follows because we always remove at least one short-braid instance from
$\w$ every time we apply $\f$.  In particular, since $m > 0$, we have that $h(\f^m( \E_{\w}^{B} ))$ is
known to be invariant under the involution by induction.

Thus, $h(\E_{\w}^c) = \overline{ h(\E_{\w}^c) }$ as well.
\end{proof}

This completes the proof of Theorem~\ref{t:mc_main}, and shows that there is a
non-recursive algorithm to compute the Kazhdan--Lusztig basis elements associated
to maximally-clustered hexagon-avoiding permutations.


\bigskip
\section{Patterns of maximally-clustered elements}\label{s:patterns}

In this section, we prove that the property of heap-avoiding the hexagon in the
maximally-clustered and freely-braided cases can be characterized by avoiding
the four 1-line patterns 
\[ \{ [46718235], [46781235], [56718234], [56781234] \}. \]
This is implicit in \cite{b-w} for the fully-commutative elements.  More
generally, we describe when it is possible to translate between heap-avoidance
and classical permutation pattern avoidance.  This also provides a methodology
for using heaps to study classical permutation pattern classes.

The definitions and results in this section generalize those in \cite[Section
11]{b-j1} to arbitrary subsets of permutations that are characterized by
classical pattern avoidance.  See \cite[Section 11]{b-j1} for results that hold
in type $D$ comparing embedded factor avoidance and classical 1-line pattern
avoidance.  The proofs from that work are very similar to those given here but
need to verified in the more general setting, so we reproduce them for
convenience.

Let $S^{P} = \bigcup_{n \geq 1} S^{P}_{n}$ denote the permutations
characterized by avoiding a set of 1-line patterns $P$.  The most important
pattern classes for this work are the maximally-clustered permutations and the
freely-braided permutations, characterized by avoiding the patterns from
\eqref{e:mc.patterns} and \eqref{e:fb.patterns}, respectively.  Given a
permutation $h$, let $S^{P}(h)$ be the subset of $S^{P}$ consisting of those
permutations that heap-avoid the single pattern $h$.  Our goal is to find a set
of 1-line patterns $Q$ such that $S^{P}(h) = S^{Q}$.  Note that this is not
always possible, as demonstrated in Example 11.1 of \cite{b-j1}.

If $r(h)$ is the rank of the symmetric group containing $h$, then we let
$U^{P}(h)$ denote the set of all elements in $S^{P}_{r(h)}$ that heap-contain
$h$.  We will show that when the patterns in $U^{P}(h)$ satisfy an additional
hypothesis called the \em ideal pattern \em condition, heap-avoiding $h$ is
equivalent to avoiding the permutations of $U^{P}(h)$ as 1-line patterns.  This
set is finite since it includes only permutations from a fixed rank.  If the
support of $h$ is connected in the Coxeter graph, then the only orientation
preserving Coxeter embedding $f: S_{r(h)} \rightarrow S_{r(h)}$ is the
identity, so the elements of $U^{P}(h)$ are precisely the elements of
$S^{P}_{r(h)}$ that contain $h$ as a factor.  In this case, $U^{P}(h)$ is the
upper order ideal generated by $h$ in the two-sided weak Bruhat order on
$S_{r(h)}$.

\begin{proposition}\label{p:mc.embedded.to.one}
Let $w \in S^{P}$ and $h$ be a permutation.  If $w$ heap-contains $h$, then $w$
contains an element of $U^{P}(h)$ as a 1-line pattern.
\end{proposition}
\begin{proof}
We begin by choosing a commutativity class of $w$ whose heap contains a
collection of lattice points corresponding to the heap of $h$.  Highlight a
shifted copy of the heap of $h$ as a set of lattice points inside the heap of
$w$.  Then, we can build the heap of $w$ starting from the shifted copy of the
heap of $h$ by sequentially adding lattice points that are maximal or minimal
with respect to the intermediate heap poset. 

Since the heap of $h$ is a convex subposet of the heap of $w$, any linear
extension of the heap of $h$ can be extended to a linear extension of the heap
of $w$.  To be precise, let $p_1 < p_2 < \dots < p_k$ be a linear extension of
the heap poset of $w$ and suppose that the interval $p_i < p_{i+1} < \dots <
p_j$ of this linear extension consists exactly of the entries from the heap of
$h$.  If we add lattice points to the heap of $h$ in the order $p_{i-1},
p_{i-2}, \dots, p_1, p_{j+1}, p_{j+2}, \dots, p_k$ then at each step we
sequentially add lattice points that are maximal or minimal in the heap poset of
$w$ restricted to the lattice points that were added in previous steps.  If the
heap of $w$ contains multiple connected components then we may add an entry at
some stage to start the new component and this entry will be unrelated to the
points previously added.  Eventually, we add all of the lattice points and
obtain the heap of $w$.

Next, suppose that the shifted copy of the heap of $h$ occupies columns $s, s+1,
\dots, t$ in the heap of $w$.  Then, we begin with the set of strings $S = \{s,
s+1, \dots, t+1\}$ that appear in the shifted copy of the heap of $h$.  These
strings initially correspond to the 1-line pattern $h$, and we show by induction
that $S$ continues to encode a 1-line pattern from $U^{P}(h)$ as we add minimal
or maximal lattice points to the heap.  Consider the relative order of the
strings in $S$ when we add a maximal lattice point to the heap.  Since the point
is maximal, the strings being crossed by the point are adjacent.  Thus, if the
new point crosses a pair of strings that are both in $S$, then the new string
configuration on $S$ corresponds to an element in $U^{P}(h)$.  If the new point
crosses a pair of strings such that at most one is contained in $S$, then the
string configuration on $S$ is unchanged.  Similarly, the relative order of the
strings in $S$ corresponds to an element in $U^{P}(h)$ when we add a minimal
lattice point, since the strings being crossed at each stage are adjacent.

At the end of this inductive construction, $w$ contains the 1-line pattern
encoded by the strings in $S$, and the element corresponding to this 1-line
pattern heap-contains $h$.
\end{proof}

The converse of Proposition~\ref{p:mc.embedded.to.one} can fail in general, as
demonstrated in Example 11.1 of \cite{b-j1}.  However, on the special patterns
defined below a converse can be stated.

\begin{definition}\label{d:mc.ideal}
Let $p \in S^{P}$.  Then, we say that $p$ is an \em ideal pattern in $S^{P}$ \em
if for every $q \in S^{P}_{r(p)+1}$ containing $p$ as a 1-line pattern, we have
that $q$ heap-contains $p$.
\end{definition}

This finite test extends to permutations of all ranks according to the following
result.

\begin{proposition}\label{p:mc.one.to.embedded}
If $h \in S^{P}$ is an ideal pattern and $w \in S^{P}$ contains $h$ as a 1-line
pattern, then $w$ heap-contains $h$.
\end{proposition}
\begin{proof}
Consider the case that $w \in S_{r(h)+1}$.  Then by Definition~\ref{d:mc.ideal} we
have that $w$ heap-contains $h$, so we can highlight an instance of the heap of
$h$ inside the heap of $w$.

By induction, assume the proposition holds for all elements in $\cup_{k=1}^{n}
S^{P}_k$ and let $w \in S^{P}_{n+1}$.  If $w$ contains $h$ as a 1-line pattern
then $w$ contains some $h' \in S^{P}_{n}$ as a 1-line pattern, with the property
that $h'$ contains $h$ as a 1-line pattern.  Applying induction, we have that
$h'$ heap-contains $h$, and we want to show that the heap of $w$ must also
contain a copy of the heap of $h$. 

The string diagram imposed on the heap of $w$ can be obtained from the string
diagram on the heap of $h'$ by adding one additional string. The additional
string will add extra points to the heap at each crossing.  This string may cut
through the highlighted copy $C$ of the heap of $h$, but since $h$ is ideal, the
extra points that are added together with $C$ must heap-contain $h$ by
Definition~\ref{d:mc.ideal}.  Therefore, $w$ heap-contains $h$ as a subheap.
\end{proof}

Thus combining Proposition~\ref{p:mc.embedded.to.one} and
Proposition~\ref{p:mc.one.to.embedded}, we have shown the following result.

\begin{theorem}\label{t:mc.one.iff.embedded.finite}
Suppose $S^{P}(H)$ is the subset of permutations characterized by avoiding a
finite set $P$ of 1-line patterns and heap-avoiding a finite set $H$ of
permutations.  If each of the elements in $P' = \bigcup_{h \in H} U^{P}(h)$ is
an ideal pattern, then $S^{P}(H) = S^{P \union P'}$, so is characterized by
avoiding the permutations in $P \union P'$ as 1-line patterns.
\end{theorem}

\begin{corollary}\label{c:hexagon.1-line}
Let $w$ be any permutation.  Then, $w$ is freely-braided and hexagon-avoiding if
and only if $w$ avoids 
\[ \{ [3421], [4231], [4312], [4321], [46718235], [46781235], [56718234], [56781234] \} \] 
as 1-line patterns.  Also, $w$ is
maximally-clustered and hexagon-avoiding if and only if $w$ avoids
\[ \{ [3421], [4312], [4321], [46718235], [46781235], [56718234], [56781234] \} \] 
as 1-line patterns.
\end{corollary}
\begin{proof}
It is straightforward to verify that
\[ U^{P}([46718235]) = \{ [46718235], [46781235], [56718234], [56781234] \} \]
for $P \in \{ \text{maximally-clustered}, \text{freely-braided} \}$
and each of these patterns are ideal.  The corollary then follows from
Theorem~\ref{t:mc.one.iff.embedded.finite}.
\end{proof}

\begin{remark}
Theorem~\ref{t:mc.one.iff.embedded.finite} can also be applied to study
classical permutation pattern classes using heap-avoidance.  If $P$ is an upper
order ideal in 2-sided weak Bruhat order consisting of connected ideal patterns,
and $P$ has finitely many minimal elements $H$, then $S^{P} = S(H)$.
\end{remark}

%


\bigskip
\section*{Acknowledgments}

Hugh Denoncourt and Jozsef Losonczy provided valuable comments on an earlier
draft of this work.  In addition, we thank Sara Billey and Richard Green for
many useful conversations and suggestions, Greg Warrington for fast computer
code to calculate the Kazhdan--Lusztig polynomials, and Julian West for
introducing us to Olivier Guibert's code \cite{guibert} for permutation pattern
enumeration.

\appendix
\bigskip
\section{}\label{s:appendix}

The idea of the following algorithm is that each $[(m+2) 2 3 \dots (m+1)
1]$-instance in the 1-line notation for $w$ corresponds to a length $2 m + 1$
braid cluster in the reduced expression $\w$ that we build.  Since each of the
$[(m+2) 2 3 \dots (m+1) 1]$-instances contributes exactly $m$ to the number of
$[321]$-instances $N(w)$, we have that $\w$ is contracted by
Definition~\ref{d:mc}.


\begin{algorithm}{\bf (Produce a contracted expression for a maximally
	clustered permutation.)}
\begin{enumerate}
\item[(1)]  Given a maximally clustered permutation $w$ in 1-line notation,
	choose the unique $[321]$-instance in positions $(i, j, k)$ such that
	$j$ is leftmost.  This $[321]$-instance may be part of a larger $[(m+2)
	2 3 \dots (m+1) 1]$-instance corresponding to a braid cluster.
\item[(2)]  Then $i$ can always be brought to the right by a reduced sequence
	of length-decreasing moves so that $i$ is made adjacent to $j$.
\item[(3)]  At this point, we can mark positions 
\[ (\dots, i, j, \dots, h, \dots, k, \dots) \]
in the 1-line notation for $w$ where all entries weakly right of $h$ and
strictly left of $k$ have value greater than $i$.  Those entries lying strictly
left of $h$ and right of $j$ have values between those of $i$ and $j$, and
include no descents.
\item[(4)]  If there exists some other $[(m'+2) 2 3 \dots (m'+1) 1]$-instance
	that uses $k$ then it occurs in the segment between $h$ and $k$.  The
	instance can be made consecutive, and we apply a braid cluster to bring
	$k$ to the left.  Otherwise, $k$ does not participate in any other
	$[(m'+2) 2 3 \dots (m'+1) 1]$-instance and we simply move $k$ to the left
	by adjacent transpositions.
\item[(5)]  Eventually, $k$ will be moved past $h$.  We apply a braid
	cluster to undo the consecutive instance in positions $(i, j, \dots, h-1)$.
\item[(6)]  Repeat from (1), until there are no more $[(m+2) 2 3 \dots (m+1)
	1]$-instances for any $m$.
\end{enumerate}
\end{algorithm}
\begin{proof}
We will show that in a reduced fashion the algorithm makes every $[(m+2) 2 3 \dots
(m+1) 1]$ pattern instance consecutive and then applies a braid cluster to undo
the $[(m+2) 2 3 \dots (m+1) 1]$ instance.  In particular, we must never move an
entry that plays the role of $3$ past an entry that plays the role of $2$ in
some $[321]$ instance, unless we are applying a braid cluster to an instance
that has already been made consecutive.  Similarly, we cannot move an entry
that plays the role of $2$ past an entry that plays the role of $1$ in some
$[321]$ instance unless the move is part of a braid cluster.

First, observe that no entry can play the role of 2 in more than one
$[321]$-instance.  To see this, suppose that the 2 in $3 2 1$ is used in some
other instance $3' 2 1'$.  If $3'$ is not distinct from $3$, then we have a
forbidden pattern when we attempt to place $1'$ among the $321$ entries:
\[ 3 2 (1.5) 1 = [4321]; \ \  3 2 1 (1.5) = [4312]; \ \  3 2 (0.5) 1 =
[4312]; \ \ 3 2 1 (0.5) = [4321] \] 
Here we have indicated the value of $1'$ in parentheses, and the corresponding
permutation pattern where the relative values have been normalized to form a
permutation in $S_4$.  Otherwise, $3'$ is distinct from $3$ and then we have a
forbidden pattern when we attempt to place $3'$ among the $321$ entries:
\[ (2.5) 3 2 1 = [3421]; \ \  3 (2.5) 2 1 = [4321]; \ \  (3.5) 3 2 1 =
[4321]; \ \ 3 (3.5) 2 1 = [3421] \] 
Hence if an entry plays the role of 2 in any $[321]$-instance then the instance
is uniquely determined, proving (1).

A length-decreasing move in (2) is always available since otherwise we have a
forbidden pattern $3 (3.5) 2 1 = [3421]$.  By the leftmost choice in (1), we
never move $i$ past another entry that plays the role of $2$ in an instance
with $i$.  

Next, note that if any entry between $j$ and $k$ has value less than the value
at $j$, then we have a forbidden $[4312]$ or $[4321]$ instance.  Moreover, if an
entry located between $j$ and $k$ with value greater than the value at $i$
occurs to the left of an entry located between $j$ and $k$ with value less than
the value at $i$, then we have a forbidden $[3421]$-instance.  Therefore, all
entries with value greater than the value at $i$ occur to the right of all the
entries with values between those of $j$ and $i$.  Also, if there is a descent
among the entries with value less than the value at $i$ then we have a forbidden
$[4321]$-instance so these values are all increasing.  Hence, the 1-line
notation for $w$ has marked positions $(\dots, i, j, \dots, h, \dots, k, \dots)$
as described in (3).

To prove (4) note that $k$ can always be moved to the left in a length
decreasing fashion, or else we obtain a $[4312]$ pattern.  It remains to check
that we can consecutively undo any other $[(m+2) 2 3 \dots (m+1) 1]$-instance in
which $k$ participates.  Observe that $k$ cannot play the role of 3 nor 2 in
another pattern, or we obtain a forbidden $[4321]$ pattern.

Hence, we suppose that $k$ plays the role of 1 in an instance of the form
$(m+2)' 2' 3' \dots (m+1)' 1$ where $(m+2)'$ must occur in a position weakly to
the right of $h$ because there are no descents among the entries with values $<
i$.  Also, the entries $2', 3', \dots, (m+1)'$ must be consecutive, since any
entry with value $> (m+2)'$ among these yields a forbidden $[3421]$ pattern,
while any entry with value $< 1$ yields a forbidden $[4312]$ pattern.

Then the entries of our 1-line notation are of the form 
\[ 3  2  \dots  (m+2)'  \dots  2' 3' \dots (m+1)' 1 \]
and we can assume that $2' 3' \dots (m+1)' 1$ and $3 2$ have been made
consecutive by previous steps.  A length-decreasing move to make $(m+2)'$ closer
to $2'$ is always available since otherwise we have a $[3421]$ pattern.  The
only other obstruction is an entry $2''$ that participates with $(m+2)'$ in
another $[321]$-instance.  However, if $2''$ is not part of the $(m+2)' 2' 3'
\dots (m+1)' 1$ cluster, then we must have that $2''$ has value less than that
of $1$.  Hence, we obtain a $[4312]$ pattern from the $3 2 2'' 1$ entries.

It is evident that (5) and (6) can be accomplished using the previous steps of
the algorithm.

Since the permutation is assumed to be finite, and we reduce the length at each
step, the algorithm eventually terminates.  By construction, the reduced
expression we produce from the algorithm has a braid cluster for each
$[m 2 3 \dots (m-1) 1]$-instance.  Hence, it is maximally clustered.
\end{proof}



\end{document}